\documentclass[12pt,a4paper]{article}
\usepackage{enumerate}
\usepackage[active]{srcltx}
\usepackage{color}
\usepackage{graphicx}
\usepackage{multirow}
\usepackage{amsmath}
\usepackage{amsfonts,latexsym}
\usepackage{amsthm}
\usepackage{verbatim}
\usepackage{amssymb}

\newtheorem{theorem}{Theorem}[section]
\newtheorem{corollary}[theorem]{Corollary}
\newtheorem{lemma}[theorem]{Lemma}
\newtheorem{proposition}[theorem]{Proposition}

\theoremstyle{definition}

\theoremstyle{remark}
\newtheorem{remark}[theorem]{Remark}

\numberwithin{equation}{section}

\newcommand{\N}{\ensuremath{\mathbb{N}}}
\newcommand{\R}{\ensuremath{\mathbb{R}}}
\newcommand{\Z}{\ensuremath{\mathbb{Z}}}

\newcommand{\abs}[1]{\left\vert #1\right\vert}
\newcommand{\set}[1]{\left\{ #1\right\}}
\newcommand{\norm}[1]{\left\|#1\right\|}
\newcommand{\pa}[1]{\left(#1\right)}
\newcommand{\pro}[2]{\left\langle#1|#2\right\rangle}
\renewcommand{\Re}[1]{\mathrm{Re}\left(#1\right)}

\begin{document}

\markboth{J.P. Borgna, M. De Leo, C. S\'anchez de la Vega, D.
Rial}{Lie-Trotter method for abstract semilinear \\ evolution equations.}

\title{Lie-Trotter method for abstract semilinear \\ evolution equations.}

\author{
Juan Pablo Borgna\thanks{Instituto de Ciencias, Universidad Nacional
de General Sarmiento, J.M. Guti\'errez 1150 (1613) Los Polvorines,
Buenos Aires,
Argentina. Email: {\em jpborgna@ungs.edu.ar}}\,,\;%
Mariano De Leo\thanks{Instituto de Ciencias, Universidad Nacional
de General Sarmiento, J.M. Guti\'errez 1150 (1613) Los Polvorines,
Buenos Aires,
Argentina. Email: {\em mdeleo@ungs.edu.ar}}\,,\; \\
Constanza S\'anchez de la Vega\thanks{IMAS - CONICET and Instituto de Humanidad, Universidad Nacional de General Sarmiento, J.M. Guti\'errez 1150 (1613) Los Polvorines, Buenos Aires, Argentina. Email: {\em csfvega@dm.uba.ar}}\,,\;%
Diego Rial\thanks{IMAS - CONICET and Departamento de Matem\'atica, Facultad de
Ciencias Exactas y Naturales, Universidad de Buenos Aires, Ciudad
Universitaria, Pabell\'on I (1428) Buenos Aires, Argentina. Email:
{\em drial@dm.uba.ar}}%
}

\maketitle

\begin{abstract}
In this paper we present a unified picture concerning Lie-Trotter
method for solving a large class of semilinear
problems: nonlinear Schr\"odinger, Schr\"oginger--Poisson,
Gross--Pitaevskii, etc. This picture includes more general schemes
such as Strang and Ruth--Yoshida. The convergence
result is presented in suitable Hilbert spaces related with the
time regularity of the solution and is based on Lipschitz
estimates for the nonlinearity. In addition, with extra
requirements both on the regularity of the initial datum and on
the nonlinearity we show the linear convergence of the method.

\vspace{1cm}


\bigskip


\end{abstract}

\section{Introduction}

Let us consider the semilinear evolution equation

\begin{align} \label{ec-evol}
\begin{cases}
u_{t}+iAu+iB\pa{u}=0,\\
u\pa{0}=u_{0}\in\mathsf{H}_{1},
\end{cases}
\end{align}
where $A$ is a self-adjoint operator
in the Hilbert space $\mathsf{H}_{1}$ and $B:\mathsf{H}_{1} \to
\mathsf{H}_{1}$ is a locally Lipschitz map. Since a large number
of problems fall under this situation, at least we can mention the
nonlinear Schr\"odinger, Schr\"oginger--Poisson, Gross--Pitaevskii
(see \cite{C} for more details), and a large amount of
articles are devoted to the numerical study of time-splitting
methods, most of them concerning Lie-Trotter and Strang schemes,
see \cite{BS,BBD,DT,FGP,G,JL}, we shall present in this
article a unified picture of time-splitting methods. This means
that we shall show general results concerning both the order of
convergence, and the regularity required for initial data. Despite
the fact that we are mainly interested in time discretization,
note that the standard result for Lie--Trotter schemes developed
in the literature expresses that the convergence is globally
linear in the time step, we also take under consideration
discretization in space (see subsection \ref{sub: spectral methods}).
In addition, we also show that under the (weaker) assumptions made
above on the operators the method is well defined and converges in
the smaller space $\mathsf{H}_1$.

To see this we first show how to solve the problem \eqref{ec-evol}
by means of a generic time-splitting scheme. Note that any
solution of \eqref{ec-evol} verifies the fixed point integral
equation
\begin{align}
\label{duhamel}
u\pa{t}={\Phi^{A}\pa{t}}u_{0}-i\int_{0}^{t}\Phi^{A}\pa{t-t'}B\pa{u\pa{t'}}dt',
\end{align}
where $\Phi^{A}$ denotes the strongly continuous one-parameter
unitary group generated by $-iA$, this means that:
$v\pa{t}=\Phi^{A}\pa{t}v_{0}$ is the solution of the linear
problem
\begin{align}
\label{A}
\begin{cases}
v_{t}+iAv=0,\\
v\pa{0}=v_{0}.
\end{cases}
\end{align}

The following well-posedness result of \eqref{duhamel} is
well-known, for proof and details, see \cite{C}.

\begin{proposition}
Let $B$ be a locally Lipschitz map defined on the Hilbert space $\mathsf{H}_{1}$
with $B\pa{0}=0$. Then for any $u_{0}\in\mathsf{H}_{1}$ there
exists $T^*=T^*(u_0)>0$ and a unique solution $u\in
C\pa{[0,T^*(u_0)),\mathsf{H}_{1}}$ of equation \eqref{duhamel}.
Moreover, the map $T^*: \mathsf{H}_{1} \to
C\pa{[0,T^*(u_0)),\mathsf{H}_{1}}$ is lower semicontinuous, and
for any $T<T^*(u_0)$ the map $\mathsf{H}_{1} \mapsto
C\pa{[0,T],\mathsf{H}_{1}}$ given by $u_{0}\mapsto u$ is
continuous, i.e.: given $\varepsilon>0$, there exists $\delta>0$
such that if $\norm{u_{0}-\tilde{u}_{0}}<\delta$ then
$T<T^{*}\pa{\tilde{u}_{0}}$ and
$\norm{u\pa{t}-\tilde{u}\pa{t}}<\varepsilon$ for $t\in[0,T]$,
where $\tilde{u}$ is the solution  of \eqref{duhamel} with
$\tilde{u}\pa{0}=\tilde{u}_{0}$. Finally, is also valid the
blow-up alternative:
\begin{enumerate}
\item $T^{*}(u_0)=\infty$ ($u$ is globally defined).
\item $T^{*}(u_0)<\infty$ and $\lim\limits_{t\uparrow
T^{*}(u_0)}\norm{u\pa{t}}=\infty$.
\end{enumerate}
\end{proposition}

Since $B$ is a locally Lipschitz map, there exists a flow
$\Phi^{B}$, defined locally in time, generated by the problem
\begin{align}
\label{eq: flow B}
\begin{cases}
w_{t}+iB\pa{w}=0,\\
w\pa{0}=w_{0}.
\end{cases}
\end{align}
Let $\Phi$ be the flow of the equation
$-i(A+B)$ defined by $\Phi\pa{t}\pa{u_{0}}=u\pa{t}$, where $u$ is the solution of \eqref{duhamel}.
The idea of time-splitting methods is to approximate $\Phi$, the
exact flow, by combining the exact flows $\Phi^{A}$ and
$\Phi^{B}$, in the following sense: for any (small) time step
$h>0$, the discrete flow is defined by
\[
\Phi_{h}=\Phi^{B}\pa{b_{m}h}\circ\Phi^{A}\pa{a_{m}h}\circ\cdots\circ
\Phi^{B}\pa{b_{1}h}\circ\Phi^{A}\pa{a_{1}h},
\]
where the splitting scheme given by $a_1,\ldots,a_m$, $
b_1,\ldots,b_m$ verifies
$a_{1}+\cdots+a_{m}=b_{1}+\cdots+b_{m}=1$. Let us mention that for
$m=1$ (therefore $a_1=b_1=1$) we get the Lie-Trotter scheme; and
for $m=2$ and $a_1=a_2=1/2, b_1=1, b_2=0$ we get the Strang
scheme. Other Yoshida schemes (see details in \cite{Y}) are
represented similarly.

For fixed $u_0\in \mathsf{H}_{1}$ and $T<T^*(u_0)$, the
convergence result expresses that
$\{u_0,\Phi_{h}(u_0),\ldots,\Phi_{h}^n(u_0)\}$ converges in some
sense to the exact solution at time $t=kh$, i.e.
$\{u_0,\Phi(h)(u_0),\ldots,\Phi(nh)(u_0)\}$, when the time step
$h=T/n$ goes to $0$. We note that the splitting scheme given by
$a_1,\ldots, a_m$ and $b_1,\ldots, b_m$ is performed $n$ times
before reaching the value $t=T$. Clearly, the scaling $t\to T t$
allows us to restrict our attention to the normalized case $T=1$,
and this will be the case in the sequel. We therefore set
$\alpha,\beta$ as the $1$-periodic functions defined by:
\[
\alpha\pa{t}=
\begin{cases}
 2m a_{j}&\text{, if }j-1\leq m\pa{t-[t]}<j-1/2\\
 0&\text{, if }j-1/2\leq m\pa{t-[t]}<j
\end{cases}
\]

\[
\beta\pa{t}=
\begin{cases}
 0&\text{, if }j-1\leq m\pa{t-[t]}<j-1/2\\
 2m b_{j}&\text{, if }j-1/2\leq m\pa{t-[t]}<j.
\end{cases}
\]
It is, then, a straightforward computation to verify that  for
$n\in \mathbb{N}$ and $\alpha_{n}\pa{t}=\alpha\pa{n
t},\beta_{n}\pa{t}=\beta\pa{n t}$, the continuous flow generated
by the (non-autonomous) operator $-i\pa{\alpha_{n}A+\beta_{n}B}$,
denoted by  $\Phi_n$, verifies $\Phi_n(1/n)=\Phi_h$. Therefore,
the convergence (in time) of the splitting scheme is expressed as
$\Phi_n\pa{t}$ converges to $\Phi\pa{t}$ as the time step $h=1/n$
goes to 0. In what follows we shall refer to an {\em abstract
time-splitting method} when we are given a pair of $T$-periodic
functions $\alpha, \beta$.

Finally, we also take into consideration the convergence in space.
It is a common practice to solve the problem \eqref{A} by means of
spectral methods, which consists of solving the problem on a
finite dimensional invariant subspace (generated by eigenfunctions
of the linear operator $A$). Since invariant subspaces of $A$ are
not necessarily $\Phi^{B}$-invariant, the approximated solution is
projected before the application of $\Phi^{A}$; this gives the
(finite dimensional) discrete flow:
\[
\widetilde{\Phi}_{h}=\Phi^{B}\pa{b_{s}h}\circ\Phi^{A}\pa{a_{s}h}\circ P\circ\cdots\circ P\circ
\Phi^{B}\pa{b_{1}h}\circ\Phi^{A}\pa{a_{1}h}\circ P,
\]
where $P$ is the orthogonal projection onto the finite dimensional
invariant subspace.

{In a more general setting, if we take $\widetilde{\Phi}_A$ as an approximation of the exact flow $\Phi_A$, this gives the discrete flow:
\begin{equation} \label{phitilde}
\widetilde{\Phi}_{h}=\Phi^{B}\pa{b_{s}h}\circ\tilde{\Phi}^{A}\pa{a_{s}h}\circ \cdots\circ
\Phi^{B}\pa{b_{1}h}\circ\tilde{\Phi}^{A}\pa{a_{1}h}.
\end{equation}}

\subsection{Notation and Main Results}

Throughout this paper the evolution problem is given by equation \eqref{ec-evol}
\begin{align*}
\begin{cases}
&u_{t}+iAu+iB\pa{u}=0,\\
&u\pa{0}=u_{0},
\end{cases}
\end{align*}
for $u_{0}\in\mathsf{H}_{1}$, where $A$ is a self-adjoint operator
in $\mathsf{H}_{1}$, and $B:\mathsf{H}_{1} \to \mathsf{H}_{1}$ is a locally Lipschitz map. The problem under
consideration is to find the generated flow $\Phi(t)$ in a compact
interval $[0,T]$, where the solution exists.
The abstract time-splitting method to solve the evolution problem
(\ref{ec-evol}) for $t\in [0,T]$, i.e. to get the flow $\Phi(t)$, is
thus described as follows:
\begin{enumerate}
 \item Set $\alpha, \beta \in L^1_{\rm loc}$ T-periodic bounded functions with total integral
 \[\int_0^T \alpha = \int_0^T \beta =1.\]
 \item Fix $n\in \mathbb{N}$ and the step size $h_{n}=T/n$ (the choice $T=1$ shall be used in the sequel).
 \item Set the sequences $\alpha_n(t)=\alpha(nt)$ and $\beta_n(t)=\beta (nt)$.
 \item Get the flow $\Phi_{h_{n}}$ of the non-autonomous
       equation $u_t=-i\pa{\alpha_{n}A+\beta_{n}B}u$.
\end{enumerate}
Under this situation we show:

\setcounter{section}{3}
\setcounter{theorem}{0}
\begin{theorem}[Convergence]
Let $u_0\in \mathsf{H}_{1}$ and $T<T^*(u_0)$, then there
exists $n_0\in \mathbb{N}$ such that for any $n\geq n_0$, the
function $\Phi_n(t) u_0$ is defined for $t\in [0,T]$, and
$\lim\limits_{n\to \infty} \max\limits_{t\in [0,T]}\norm{u\pa{t}-u_{n}\pa{t}} = 0$.
\end{theorem}
\setcounter{section}{1}
\setcounter{theorem}{2}

In order to get the order of convergence for abstract methods some
extra regularity both on the time derivative and on the
nonlinearity is needed. The basic assumption is as follows: let
$\mathsf{H}_{0}$ be a Hilbert space such that $\mathsf{H}_{1}
\subseteq \mathsf{H}_{0}$, with continuous embedding, we asume
\begin{enumerate}
 \item The solution $u$ of \eqref{duhamel} verifies $u\in W^{1,\infty}\pa{[0,T],\mathsf{H}_{0}}$.
 \item There exists a bounded map
$B':\mathsf{H}_{1} \mapsto \mathcal{B}\pa{\mathsf{H}_{0}}$ such that, for $\varepsilon>0$ and $u\in \mathsf{H}_{1}$, the estimate
\begin{align*}
\norm{B\pa{u+w}-B\pa{u}-B'\pa{u}w}_{\mathsf{H}_{0}}\leq\varepsilon\norm{w}_{\mathsf{H}_{0}}
\end{align*}
holds for some $\delta >0$ and for any $w\in H_1$ with $\norm{w}_{\mathsf{H}_{1}}<\delta$.
\end{enumerate}

\setcounter{section}{3}
\setcounter{theorem}{8}
\begin{theorem}[Local error]
Let $u_{0}\in \mathsf{H}_{1}$ and $T<T^{*}(u_0)$, then
there exists a constant $C>0$ and $n_{0}\in\mathbb{N}$ such that
for $n\geq n_{0}$, the following estimate holds for the time step
$h_n=T/n$
\[
\norm{\Phi\pa{h_{n}} u_0-\Phi_n \pa{h_{n}}
u_{0}}_{\mathsf{H}_{0}}\leq C h_{n}^{2}.
\]
\end{theorem}

\begin{theorem}[Global error]
Let $u_{0}\in \mathsf{H}_{1}$ and $T<T^{*}(u_0)$, then
there exists a constant $C>0$ and $n_{0}\in\mathbb{N}$ such that,
for $n\geq n_{0}:$
\[
\max\limits_{0\leq k \leq n}\norm{\Phi\pa{kh_{n}} u_0-\Phi_n
\pa{kh_{n}} u_{0}}_{\mathsf{H}_{0}}\leq C h_{n}.
\]
\end{theorem}
\setcounter{section}{1}

\section{Auxiliary Results}

This section is devoted to present some basic results that we use
to prove the convergence theorems. We start with the following
notion. We say that a sequence $\{\alpha_{n}\}_{n\in \mathbb{N}}$
of functions in $L_{\text{loc}}^{1}\pa{\R}$ converges weakly to
$\alpha\in L_{\text{loc}}^{1}\pa{\R}$, denoted by
$\alpha_{n}\rightharpoonup \alpha$, if for any compact interval
$I\subset\R$ and $\theta\in C\pa{I}$, the following estimate
holds
\[
\lim_{n\to\infty}\int_{I}\alpha_{n}\pa{t}\theta\pa{t}dt=\int_{I}\alpha\pa{t}\theta\pa{t}dt.
\]
\begin{lemma}
\label{AA} Let $\alpha_{n},\alpha,\bar{\alpha}\in
L_{\text{loc}}^{1}\pa{\R}$, $n\in\mathbb{N}$, such that
$\alpha_{n}\rightharpoonup\alpha$ and
$\abs{\alpha_{n}}\leq\bar{\alpha}$. Then for any $\theta\in
C\pa{[0,T]}$ the sequence
$\Theta_{n}\pa{t}=\displaystyle{\int_{0}^{t}}\alpha_{n}\pa{t'}\theta\pa{t'}dt'$
converges uniformly to
$\Theta\pa{t}=\displaystyle{\int_{0}^{t}}\alpha\pa{t'}\theta\pa{t'}dt'$,
on $[0,T]$.
\end{lemma}
\begin{proof}
Suppose $\Theta_{n}$ does not converge to $\Theta$ uniformly, then
there exists $\varepsilon>0$ and a subsequence $\Theta_{n_{k}}$
such that $\max\limits_{0\leq t\leq
T}\abs{\Theta\pa{t}-\Theta_{n_{k}}\pa{t}}\geq\varepsilon$. Using
the estimate
\[
\abs{\Theta_{n_{k}}\pa{t}}\leq\max\limits_{0\leq t\leq T}\abs{\theta\pa{t}}\norm{\bar{\alpha}}_{L^{1}\pa{[0,T]}},
\]
we have that the sequence $\set{\Theta_{n_{k}}}_{n\geq 1}$ is
uniformly bounded in $C\pa{[0,T]}$. A similar argument allows us
to conclude that the sequence $\set{\Theta_{n_{k}}}_{n\geq 1}$ is
equicontinuous. By Arzel\'a-Ascoli theorem, we obtain that (a
subsequence of) $\Theta_{n_{k}}$ converges uniformly to
$\Theta^{*}\neq\Theta$ on $[0,T]$. But $\Theta_{n_{k}}$ converges
pointwise to $\Theta$, which is a contradiction. This finishes the
proof.
\end{proof}

For any real valued function $\alpha\in
L_{\text{loc}}^{1}\pa{\R}$, we define the propagator operator
$\Phi^{A,\alpha}\pa{t_{1},t_{0}}=\Phi^{A}\pa{\tau\pa{t_{1},t_{0}}}$,
where
$\tau\pa{t_{1},t_{0}}=\displaystyle{\int_{t_{0}}^{t_{1}}}\alpha\pa{t}dt$.
It is clear that the propagator $\Phi^{A,\alpha}\pa{t_{1},t_{0}}$
verifies:
\begin{enumerate}
\item $\Phi^{A,\alpha}\pa{t_{0},t_{0}}=I$.
\item $\Phi^{A,\alpha}\pa{t_{2},t_{0}}=\Phi^{A,\alpha}\pa{t_{2},t_{1}}\Phi^{A,\alpha}\pa{t_{1},t_{0}}$.
\item If $u\in D\pa{A}$, then
$\partial_{t}\Phi^{A,\alpha}\pa{t,t_{0}}u=-i\alpha\pa{t}A\Phi^{A,\alpha}\pa{t,t_{0}}u$.
\end{enumerate}
Observe that if $u_{0}\in D\pa{A}$, then
$u\pa{t}=\Phi^{A,\alpha}\pa{t,0}u_{0}$ is the solution of the
linear evolution Cauchy problem $iu_{t}=\alpha(t)\,A u$ with
initial condition $u\pa{0}=u_{0}$.

\begin{proposition}
\label{Pr: unif convergence} Let $\{\alpha_{n}\}_{n\in
\mathbb{N}}$ be a sequence of real valued functions in
$L_{\text{loc}}^{1}\pa{\R}$ such that $\alpha_{n}\rightharpoonup
1$, then $\Phi^{A,n}\pa{t,t'}=\Phi^{A,\alpha_{n}}\pa{t,t'}$
converges strongly to $\Phi^{A}\pa{t-t'}$. Moreover, if
$\abs{\alpha_{n}}\leq\bar{\alpha}\in L_{\text{loc}}^{1}\pa{\R}$,
then the convergence is uniform for $t',t$ on bounded intervals.
\end{proposition}
\begin{proof}
Let $I\subseteq \R$ be a compact interval and $\tau_{n}:I\times
I\to\R$ defined by $\alpha_{n}$. Since $\alpha_{n}\rightharpoonup
1$, we have $\tau_{n}\pa{t,t'}\to t-t'$, thus
$\lim\limits_{n\to\infty}\Phi^{A,n}\pa{t,t'}u=\Phi^{A}\pa{t-t'}u$.
If $\abs{\alpha_{n}}\leq\bar{\alpha}$, from Lemma \ref{AA} it
follows that the sequence $\tau_{n}(t,t')$ converges to $t-t'$
uniformly on $I\times I$. For any $u\in D\pa{A}$,  the estimate
\[
\norm{\Phi^{A,n}\pa{t,t'}u-\Phi^{A}\pa{t-t'}u}\leq\abs{\tau_{n}\pa{t,t'}-\pa{t-t'}}\norm{Au},
\]
is verified. Since $D(A)$ is dense in $\mathsf{H}_{1}$, using an
$\varepsilon/3$ argument we finish the proof.
\end{proof}

\begin{lemma}
\label{Le: piecewise constant approximation} Let $v\in
C\pa{[0,T],\mathsf{H}_{1}}$ and $\varepsilon>0$. Then there exist
$\theta_{j}\in C\pa{[0,T]}$ and $z_{j}\in\mathsf{H}_{1}$, $0\leq
j\leq m$, such that the function
\begin{align}\label{def_z}
z\pa{t}=\sum\limits_{0\leq j\leq m}\theta_{j}\pa{t}z_{j}
\end{align}
satisfies $\max\limits_{t\in
[0,T]}\norm{v\pa{t}-z\pa{t}}<\varepsilon$.
\end{lemma}

\begin{proof}
Let $\delta>0$ be such that
$\norm{v\pa{t}-v\pa{t'}}<\varepsilon/2$ if $\abs{t-t'}<\delta$,
and let $t_{-1}<t_{0}=0<t_{1}<\cdots<t_{m}=T<t_{m+1}$ be a
partition with $t_{j}-t_{j-1}<\delta$. Let also $\theta_{j}\in
C\pa{I}$ be such that $0\leq\theta_{j}\leq 1$, $\sum\limits_{0\leq
j\leq m}\theta_{j}=1$ and
$\mathrm{supp}\pa{\theta_{j}}\subset\pa{t_{j-1},t_{j+1}}$. Taking
$z_{j}=v\pa{t_{j}}$  we have for $t\in[t_{j-1},t_{j}]$

\begin{align*}
\norm{v\pa{t}-z\pa{t}}=&
\norm{\pa{\theta_{j-1}\pa{t}+\theta_{j}\pa{t}}v\pa{t}-
\theta_{j-1}\pa{t}z_{j-1}-{\theta_{j}\pa{t}z_{j}}}\\
\leq&\norm{v\pa{t}-v\pa{t_{j-1}}}+\norm{v\pa{t}-v\pa{t_{j}}}.
\end{align*}
Since $\abs{t-t_{j-1}},\abs{t-t_{j-1}}<\delta$, the proof is
finished.
\end{proof}

\begin{corollary}
\label{Co: weak convergence}

Let $\beta_{n}$ be a sequence of real valued functions in
$L_{\text{loc}}^{1}\pa{\R}$ such that $\beta_{n}\rightharpoonup
0$ with $\abs{\beta_{n}}\leq\bar{\beta}\in
L_{\text{loc}}^{1}\pa{\R}$, and let $v\in
C\pa{[0,T],\mathsf{H}_{1}}$. Define $V_n(t)$ as follows
\begin{align}\label{def_Vn}
  V_{n}\pa{t}=\displaystyle{\int_{0}^{t}}\beta_{n}\pa{t'}v\pa{t'}dt'
\end{align}
Then $V_{n}\in C\pa{[0,T],\mathsf{H}_{1}}$ and
$\lim\limits_{n\to\infty}\max\limits_{t\in
[0,T]}\norm{V_{n}\pa{t}}=0$.
\end{corollary}

\begin{proof}
Let $\varepsilon>0$ and let $z(t)$ be the function given by Lemma
\ref{Le: piecewise constant approximation}. We define
\[
Z_{n}\pa{t}=\displaystyle{\int_{0}^{t}}\beta_{n}\pa{t'}z\pa{t'}dt'=\sum\limits_{0\leq
j\leq m}\Theta_{j,n}\pa{t}z_{j},
\]
where
$\Theta_{j,n}\pa{t}=\displaystyle{\int_{0}^{t}}\beta_{n}\pa{t'}\theta_{j}\pa{t'}dt'$.
From Lemma \ref{AA},
$\lim\limits_{n\to\infty}\max\limits_{t\in
[0,T]}\norm{Z_{n}\pa{t}}=0$. On the other hand, from Lemma
\ref{Le: piecewise constant approximation} we have
$\max\limits_{t\in
[0,T]}\norm{V_{n}\pa{t}-Z_{n}\pa{t}}\leq\varepsilon\norm{\bar{\beta}}_{L^{1}\pa{[0,T]}}$
which proves the result.
\end{proof}

\begin{corollary}
\label{corolario de Pr: unif convergence} Let $v\in
C\pa{I,\mathsf{H}_{1}}$ and $\{\alpha_{n}\}_{n\in \mathbb{N}}$ a
sequence of real valued functions in $L_{\text{loc}}^{1}\pa{\R}$
such that $\alpha_{n}\rightharpoonup 1$ and
$\abs{\alpha_{n}}\leq\bar{\alpha}\in L_{\text{loc}}^{1}\pa{\R}$.
Then $\Phi^{A,n}\pa{t,t'}v\pa{t'}$ converges uniformly to $\Phi^{A}\pa{t-t'}v\pa{t'}$ on $I\times I$.
\end{corollary}
\begin{proof}
Let $z(t)$ be as in Lemma \ref{Le: piecewise constant
approximation}, then
\begin{align*}
\pa{\Phi^{A,n}\pa{t,t'}-\Phi^{A}\pa{t-t'}}v\pa{t'}=&\,\Phi^{A,n}\pa{t,t'}\pa{v\pa{t'}-z\pa{t'}}\\
&+\,\Phi^{A}\pa{t-t'}\pa{v\pa{t'}-z\pa{t'}}\\
&+\,\pa{\Phi^{A,n}\pa{t,t'}-\Phi^{A}\pa{t-t'}}z\pa{t'}.
\end{align*}

Since $\Phi^{A,n}\pa{t,t'},\Phi^{A}\pa{t-t'}$ are unitary
operators, the first and the second term on the right-hand side
are bounded by $\varepsilon$. From definition of $z$, it is easy
to see that
\begin{align*}
\norm{\pa{\Phi^{A,n}\pa{t,t'}-\Phi^{A}\pa{t-t'}}z\pa{t'}}&
\leq\max_{1\leq j\leq m}\max_{t'\in I}\abs{\theta_{j}\pa{t'}}\times\\
&\sum_{1\leq j\leq
m}\norm{\pa{\Phi^{A,n}\pa{t,t'}-\Phi^{A}\pa{t-t'}}z_{j}}.
\end{align*}
Using Proposition \ref{Pr: unif convergence}, we obtain the
result.
\end{proof}

Let $\beta$ be a bounded, $1$-periodic, function. For $n\in
\mathbb{N}$ we define $\beta_{n}\pa{t}=\beta\pa{n t}$, we note
that $\beta_{n}\rightharpoonup\langle
\beta\rangle=\displaystyle{\int_{0}^{1}}\beta\pa{t}dt$. Then,
under additional hypotheses on $v$, we obtain an estimate for the
order of convergence in Corollary \ref{Pr: unif convergence}.

\begin{lemma}
\label{Le: vdif} Let $v\in W^{1,\infty}\pa{[0,h_{n}],H}$, $V_n(t)$ be given by \eqref{def_Vn} and
$w_n=V_n(h_{n})$. If $\langle \beta\rangle=0$ or $v\pa{0}=0$, then
$w_n$ satisfies
$|w_n|\leq\dfrac{1}{2}\norm{\beta}_{L^{\infty}}\norm{v_{t}}_{L^{\infty}\pa{[0,h_{n}],\mathsf{H}_{1}}}h_{n}^{2}$.
\end{lemma}

\begin{proof}
Using $v\pa{t}=v\pa{0}+\displaystyle{\int_{0}^{t}}v_{t}\pa{t'}dt'$, we obtain
\begin{align*}
w_{n}=&\langle \beta\rangle v\pa{0} h_{n}+
\int_{0}^{h_{n}}\int_{0}^{t}\beta_{n}\pa{t}v_{t}\pa{t'}dt'dt\\
=&\int_{0}^{h_{n}}\int_{0}^{t}\beta_{n}\pa{t}v_{t}\pa{t'}dt'dt,
\end{align*}
then $\norm{V_{n}\pa{h_{n}}}\leq 
{\displaystyle{\int_{0}^{h_{n}}\int_{0}^{t}}\abs{\beta_{n}\pa{t}}\norm{v_{t}\pa{t'}}dt'dt}$
and an easy estimation implies the result.
\end{proof}

\section{Main Results}

\subsection{Convergence in $\mathsf{H}_{1}$}

Let $\{\alpha_{n}\}_{n\in \mathbb{N}},\{\beta_{n}\}_{n\in \N}$ be
two sequences of real valued functions in
$L_{\text{loc}}^{1}\pa{\R}$ such that
$\alpha_{n},\beta_{n}\rightharpoonup 1$,
$\abs{\alpha_{n}}\leq\bar{\alpha}$ and
$\abs{\beta_{n}}\leq\bar{\beta}$, with
$\bar{\alpha},\bar{\beta}\in L_{\text{loc}}^{1}\pa{\R}$. For
$n\in \N$ we consider the {\em approximated} evolution problem,
\begin{align}
\label{non-autonomous}
\begin{cases}
  iw_t+\left(\alpha_n A+\beta_n B \right)w=0\\
  w(0)=u_0
\end{cases}
\end{align}
related with the abstract splitting scheme defined by these
sequences, and we denote by $\Phi_n$ the related flow. (The exact
flow will be denoted by $\Phi$.) Let $u_0 \in \mathsf{H}_{1}$ be
given and let $u_n=\Phi_n u_0$ be the solution of the problem
\eqref{non-autonomous}, we recall below the integral expression
for $u_n$
\begin{align}
\label{duhamel generalizado}
u_{n}\pa{t}=\Phi^{A,n}\pa{t,0}u_{0}-i\int_{0}^{t}\beta_{n}\pa{t'}\Phi^{A,n}\pa{t,t'}B\pa{u_{n}\pa{t'}}dt'\,.
\end{align}

We are now in position to give the first result concerning the
uniform convergence of $\Phi_n(t) u_0$ to $\Phi(t) u_0$ for $t\in
[0,T]$ and for any $u_0\in \mathsf{H}_{1}$.

\begin{theorem}[Convergence]
\label{Th: convergence on H1}
Let $u_0\in \mathsf{H}_{1}$ and $T<T^*(u_0)$, then there
exists $n_0\in \mathbb{N}$ such that for any $n\geq n_0$, the
function $\Phi_n(t) u_0$ is defined for $t\in [0,T]$, and
$\lim\limits_{n\to \infty} \max\limits_{t\in [0,T]}\norm{u\pa{t}-u_{n}\pa{t}} = 0$.
\end{theorem}

\begin{proof}
For $t<\min\set{T,T^{*}_{n}(u_0)}$, we write
\begin{align}
\label{u-un}
u\pa{t}-u_{n}\pa{t}=&I_{1,n}\pa{t}{-i}\left(\Phi^{A}\pa{t}I_{2,n}\pa{t}+I_{3,n}\pa{t}+I_{4,n}\pa{t}\right),
\end{align}
where
\begin{align*}
I_{1,n}\pa{t}=&\pa{\Phi^{A}\pa{t}-\Phi^{A,n}\pa{t,0}}u_{0},\\
I_{2,n}\pa{t}=&\int_{0}^{t}\pa{1-\beta_{n}\pa{t'}}\Phi^{A}\pa{-t'}B\pa{u\pa{t'}}dt',\\
I_{3,n}\pa{t}=&\int_{0}^{t}\beta_{n}\pa{t'}\pa{\Phi^{A}\pa{t-t'}-\Phi^{A,n}\pa{t,t'}}B\pa{u\pa{t'}}dt',\\
I_{4,n}\pa{t}=&\int_{0}^{t}\beta_{n}\pa{t'}\Phi^{A,n}\pa{t,t'}\pa{B\pa{u\pa{t'}}-B\pa{u_{n}\pa{t'}}}dt'.
\end{align*}
We shall prove that $I_{j,n}\pa{t}\to 0$ as $n\to\infty$ uniformly on $[0,T]$.

From Proposition \ref{Pr: unif convergence} we get
$\lim\limits_{n\to\infty}\max\limits_{t\in [0,T]}\norm{I_{1,n}\pa{t}}=0$. From
Corollary \ref{Co: weak convergence} we deduce
$\lim\limits_{n\to\infty}\max\limits_{t\in [0,T]}\norm{I_{2,n}\pa{t}}=0$.

For $j=3$ we have the estimate
\begin{align*}
\norm{I_{3,n}\pa{t}}\leq\norm{\bar{\beta}}_{L^{1}\pa{[0,T]}}\max_{t,t'\in [0,T]}
\norm{\pa{\Phi^{A}\pa{t-t'}-\Phi^{A,n}\pa{t,t'}}B\pa{u\pa{t'}}}.
\end{align*}
Using Corollary \ref{corolario de Pr: unif convergence} we obtain
$\lim\limits_{n\to\infty}\norm{I_{3,n}\pa{t}}=0$.

Let $R=\max\limits_{t\in[0,T]}\norm{u\pa{t}}$, and let $L$ be some Lipschitz constant of $B$ on the ball of radius
$2R$ centered at the origin. Then there exists $n_{0}\in\N$ such that for $n\geq n_{0}$ is valid the estimate
\[
\max\limits_{t\in [0,T]}\sum_{j=1}^{3}\norm{I_{j,n}\pa{t}}\leq\varepsilon\exp\pa{-L\norm{\bar{\beta}}_{L^1\pa{[0,T]}}}.
\]
Thus, we have
\begin{align*}
\norm{u\pa{t}-u_{n}\pa{t}}\leq&\,\varepsilon\exp\pa{-L\norm{\bar{\beta}}_{L^1\pa{[0,T]}}}\\
&+L\int_{0}^{t}\bar{\beta}\pa{t'}\norm{u\pa{t'}-u_{n}\pa{t'}}dt',
\end{align*}
from Gronwall inequality we obtain $\norm{u\pa{t}-u_{n}\pa{t}}\leq\varepsilon$,
and then $T<T^{*}_{n}(u_0)$. This finishes the proof.
\end{proof}

\subsection{Error estimate}

In this section we obtain local and global in time error estimates
for general time-splitting methods. These results are optimal for
Lie-Trotter schemes, whose local convergence in the whole space is
quadratic in the time step. Let $\alpha,\beta$ be $1$-periodic,
bounded functions, with $\langle \alpha\rangle=\langle
\beta\rangle=1$, and set $\alpha_{n}\pa{t}=\alpha\pa{nt}$,
$\beta_{n}\pa{t}=\beta\pa{nt}$, with $h=1/n\downarrow 0$. We
recall that, under this situation
$\alpha_{n},\beta_{n}\rightharpoonup 1$. In order to get these
error estimates we impose some regularity both on the time
derivative of the solution and on the nonlinearity $B$, which is
accomplished as follows.
We consider a Hilbert space $\mathsf{H}_{0}$ such that
$\mathsf{H}_{1}$ is continuously embedded in $\mathsf{H}_{0}$,
and there exists a self-adjoint extension of the operator $A:D\to\mathsf{H}_{0}$ with
$\mathsf{H}_{1}\subseteq D$.
We can see that for $u_{0}\in\mathsf{H}_{1}$, the solution $u$ of \eqref{duhamel} or
\eqref{duhamel generalizado} verifies $u\in
W^{1,\infty}\pa{[0,T],\mathsf{H}_{0}}$. We also assume that there
exists a map $B':\mathsf{H}_{1} \mapsto
\mathcal{B}\pa{\mathsf{H}_{0}}$ such that for $R,\varepsilon>0$, it can be chosen $C,\delta>0$ verifying
\begin{subequations}
\label{db}
\begin{align}
&\norm{B'(u)}_{\mathcal{B}\pa{\mathsf{H}_{0}}}\leq C,\\
&\norm{B\pa{u+w}-B\pa{u}-B'\pa{u}w}_{\mathsf{H}_{0}}\leq \varepsilon\norm{w}_{\mathsf{H}_{0}},
\end{align}
\end{subequations}
for $u,w\in H_1$, $\norm{u}_{\mathsf{H}_{1}}\leq R$ and $\norm{w}_{\mathsf{H}_{1}}<\delta$.

From conditions \eqref{db} it is clear that for $R>0$, there exists $L>0$ such that
$\norm{B\pa{u}-B\pa{v}}_{\mathsf{H}_{0}}\leq L\norm{u-v}_{\mathsf{H}_{0}}$
for any $u,v\in\mathsf{H}_{1}$ with $\norm{u}_{\mathsf{H}_{1}},\norm{v}_{\mathsf{H}_{1}}\leq R$.
Let $u_{0},\tilde{u}_{0}\in\mathsf{H}_{1}$, $T<\min\set{T^{*}\pa{u_{0}},T^{*}\pa{\tilde{u}_{0}}}$,
$\varepsilon>0$ and $R>0$ such that
\[
\norm{\Phi\pa{t}u_{0}}_{L^{\infty}\pa{[0,T],\mathsf{H}_{1}}},\norm{\Phi\pa{t}\tilde{u}_{0}}_{L^{\infty}\pa{[0,T],\mathsf{H}_{1}}}\leq R
\]
since $\Phi^{A}\pa{t-t'}$ ($\Phi^{A,n}\pa{t,t'}$) is an unitary operator of $\mathsf{H}_{0}$, we deduce that
\[
\norm{\Phi\pa{t}u_0-\Phi\pa{t}\tilde{u}_0}_{\mathsf{H}_{0}}\leq\norm{u_{0}-\tilde{u}_{0}}_{\mathsf{H}_{0}}
+L\displaystyle{\int_{0}^{t}}\norm{\Phi\pa{t'}u_0-\Phi\pa{t'}\tilde{u}_0}_{\mathsf{H}_{0}}dt'.
\]
Therefore, we have the estimate
\begin{align}
\label{eq:lipPhiK}
\norm{\Phi\pa{t}u_{0}-\Phi\pa{t}\tilde{u}_{0}}_{\mathsf{H}_{0}}\leq
e^{Lt}\norm{u_{0}-\tilde{u}_{0}}_{\mathsf{H}_{0}}.
\end{align}

We now define for a fixed $T>0$ the space
$\mathrm{X}_{T}=C\pa{[0,T],\mathsf{H}_{1}}\cap
W^{1,\infty}\pa{[0,T],\mathsf{H}_{0}}$.
Since $B$ is a locally Lipschitz map and conditions \eqref{db} we can see that
$u\mapsto B\circ u$ is a well--defined bounded map in $\mathrm{X}_{T}$ and
$\pa{B\circ u}_{t}=B'\pa{u}u_{t}$.

The following lemma deals with local nonlinearities.

\begin{lemma}[Local nonlinearities]
\label{ex:f}
Let $f:\mathbb{C}\to\mathbb{C}$ be a smooth map in the real sense, (i.e.: if $f=f^{\pa{r}}+if^{\pa{i}}$,
then the map $\pa{\xi,\eta}\mapsto
\pa{f^{\pa{r}}\pa{\xi+i\eta},f^{\pa{i}}\pa{\xi+i\eta}}$ is smooth
on $\R^{2}$). Let also $\mathsf{H}_{1}=H^{s}\pa{\R^{d}}$, with
$s>d/2$, and $\mathsf{H}_{0}=L^{2}\pa{\R^{d}}$. Then
$B:\mathsf{H}_{1} \mapsto \mathsf{H}_{1}$ given by $B(u)=f(u)$ is
a well-defined map, in addition $B':\mathsf{H}_{1} \mapsto
\mathcal{B}\pa{\mathsf{H}_{0}}$ given by $B'(u)(v)=f'(u)v$ is
well-defined and verifies \eqref{db}.
\end{lemma}

\begin{proof}
From Schauder lemma (see Theorem 6.1 in \cite{R}), for
$s>d/2$, it follows that $B: H^{s}\pa{\R^{d}}\mapsto
H^{s}\pa{\R^{d}}$ is a well-defined, locally Lipschitz map.
Taking norm in the identity
\[
f'\pa{u}.w=\pa{f_{\xi}^{\pa{r}}\pa{u}w^{\pa{r}}+f_{\eta}^{\pa{r}}\pa{u}w^{\pa{i}}}+
i\pa{f_{\xi}^{\pa{i}}\pa{u}w^{\pa{r}}+f_{\eta}^{\pa{i}}\pa{u}w^{\pa{i}}},
\]
we obtain $\norm{B'\pa{u}w}_{L^{2}\pa{\R^{d}}}\leq
C\pa{\norm{u}_{L^{\infty}\pa{\R^{d}}}}\norm{w}_{L^{2}\pa{\R^{d}}}$,
with $C\pa{R}=\max\limits_{\abs{u}\leq R}\abs{f'\pa{u}}$. \linebreak Using
$\abs{f\pa{u+w}-f\pa{u}-f'\pa{u}.w}<\varepsilon\abs{w}$ if
$\abs{u}\leq R$ and $\abs{w}<\delta$, we get the required
inequality. This finishes the proof.
\end{proof}


%

In order to add Hartree-type nonlinearities we first collect some
useful estimates.

\begin{lemma}
\label{Le: hartree}
Let $W_1 \in L^{\infty}\pa{\R^d}, W_2 \in
L^{p}\pa{\R^d}$, with $p\geq 2$, $p>d/4$. Let also $u\in
H^{s}\pa{\R^d}$, with $s>d/2$, and $v\in L^2\pa{\R^d}$. Then the following
estimates do hold, with $C$ depending only on $s$:
  \begin{itemize}
    \item[(i)] $\norm{W_1\ast\Re{u^*v}}_{L^\infty\pa{\R^d}} \leq
    \norm{W_1}_{L^\infty\pa{\R^d}} \norm{v}_{L^2\pa{\R^d}}
    \norm{u}_{L^2\pa{\R^d}}$

    \item[(ii)] $\norm{W_2\ast\Re{u^*v}}_{L^\infty\pa{\R^d}} \leq
    C \norm{W_2}_{L^p\pa{\R^d}} \norm{v}_{L^2\pa{\R^d}}
    \norm{u}_{H^s\pa{\R^d}}^\theta \norm{u}_{L^2\pa{\R^d}}^{1-\theta}$

    \item[(iii)] $\norm{W_1\ast|u|^2}_{L^\infty\pa{\R^d}} \leq
    \norm{W_1}_{L^\infty\pa{\R^d}} \norm{u}_{L^2\pa{\R^d}}^2$

    \item[(iv)] $\norm{W_2\ast|u|^2}_{L^\infty\pa{\R^d}} \leq
    C \norm{W_2}_{L^p\pa{\R^d}} \norm{u}_{H^s\pa{\R^d}}^{2\theta} \norm{u}_{L^2\pa{\R^d}}^{2(1-\theta)}$
  \end{itemize}
\end{lemma}

\begin{proof}
  Estimates (i) and (iii) follows immediately from Young and H\"older
  inequalities, while estimates (ii) and (iv) also uses Gagliardo-Nirenberg inequality.
\end{proof}

\begin{lemma}[Hartree-type nonlinearities]
\label{ex: W}
Let $W\in L^{\infty}\pa{\R^d}+L^p\pa{\R^d}$, with
$p\geq 2$, $p>d/4$, let $\mathsf{H}_{1}=H^s\pa{\R^d}$, with
$s>d/2$,  and $\mathsf{H}_{0}=L^2\pa{\R^d}$. Then
$B:\mathsf{H}_{1} \mapsto \mathsf{H}_{1}$, with $B(u)=\pa{W\ast
|u|^2}u$ is a well-defined map, in addition the map
$B':\mathsf{H}_{1} \mapsto \mathcal{B}\pa{\mathsf{H}_{0}}$ given
by $B'(u)(v)=\pa{W\ast \abs{u}^2}v+2\pa{W\ast \Re{u^*v}}u$ is
well-defined and verifies estimate (\ref{db}).
\end{lemma}

\begin{proof}
Since
\begin{align*}
B(u+v)-B(u)=&\pa{W\ast \abs{u}^2}v+2\pa{W\ast \Re{u^*v}}u\\
&+2\pa{W\ast \Re{u^*v}}v+ \pa{W\ast \abs{v}^2} \pa{u+v},
\end{align*}
the linear term is given by $B'(u)(v)=\pa{W\ast \abs{u}^2}v+2\pa{W\ast\Re{u^*v}}u$.
The estimate  \eqref{db} follows directly from Lemma \ref{Le: hartree}.
\end{proof}

\begin{theorem}[Local error]
\label{Th: local convergence}
Let $u_{0}\in \mathsf{H}_{1}$ and $T<T^{*}(u_0)$, then
there exists a constant $C>0$ and $n_{0}\in\mathbb{N}$ such that
for $n\geq n_{0}$, the following estimate holds for the time step
$h_n=T/n$
\[
\norm{\Phi\pa{h_{n}} u_0-\Phi_n \pa{h_{n}}
u_{0}}_{\mathsf{H}_{0}}\leq C h_{n}^{2}.
\]
\end{theorem}

\begin{proof}
Replacing $t=h_{n}$ in Eq. \eqref{u-un} and using that $\Phi^{A}\pa{h_{n}}$
are unitary operators, we see that it is sufficient to show the estimates
$\norm{I_{j,n}\pa{h_{n}}}_{\mathsf{H}_{0}}\leq C h_{n}^{2}$, where $I_{j,n}$ are defined as in Theorem \ref{Th: convergence on H1}.
Since $\langle \alpha\rangle=1$, we have
\[
I_{1,n}\pa{h_{n}}=\Phi^{A}\pa{h_{n}}-\Phi^{A,n}\pa{h_{n},0}=
\Phi^{A}\pa{h_{n}}-\Phi^{A}\pa{h_{n}\langle \alpha\rangle}=0.
\]
From Theorem \ref{Th: convergence on H1}, there exists $n_{0}\in\N$ such that for $n\geq n_{0}$ it holds $T^{*}_{n}>T$ and
$\max\limits_{t\in[0,T]}\norm{u_{n}\pa{t}}<\max\limits_{t\in[0,T]}\norm{u\pa{t}}+1=R$.
Setting $v^{\pa{2}}\pa{t}=\Phi^{A}\pa{-t}B\pa{u\pa{t}}$,
it is clear that $v^{\pa{2}}\in\mathrm{X}_{T}$ and
\[
v_{t}^{\pa{2}}\pa{t}=\Phi^{A}\pa{-t}\pa{iAB\pa{u\pa{t}}+{\left(B\pa{u\pa{t}}\right)}_{t}},
\]
from where it follows the estimate $\norm{v_{t}^{\pa{2}}}_{L^{\infty}\pa{[0,h_n],\mathsf{H}_{0}}}\leq C\pa{R}$.

{Using that
\begin{align*}
I_{2,n}\pa{h_{n}}=\int_0^{h_n} \pa{1-\beta_n(s)} v^{\pa{2}}\pa{s} ds
\end{align*}}
and since $\langle 1-\beta\rangle=0$, from Lemma \ref{Le: vdif} we deduce
\begin{align*}
\norm{I_{2,n}\pa{h_{n}}}_{\mathsf{H}_{0}}
\leq C\pa{R}\pa{1+\norm{\beta}_{L^{\infty}}}h_{n}^{2}.
\end{align*}
We set $v^{\pa{3}}\pa{t}=\pa{\Phi^{A}\pa{h_{n}-t}-\Phi^{A,n}\pa{h_{n},t}}B\pa{u\pa{t}}$.
It is clear that $v^{\pa{3}}\in\mathrm{X}_{T}$, $v^{\pa{3}}\pa{0}=0$, and
\begin{align*}
v^{\pa{3}}_{t}\pa{t}=&i\pa{\Phi^{A}\pa{h_{n}-t}-\alpha_{n}\pa{t}\Phi^{A,n}\pa{h_{n},t}}AB\pa{u\pa{t}}\\
&+\pa{\Phi^{A}\pa{h_{n}-t}-\Phi^{A,n}\pa{h_{n},t}}{\pa{B\pa{u\pa{t}}}}_{t}.
\end{align*}
Taking norms, we deduce the estimate $\norm{v^{\pa{3}}_{t}}_{L^{\infty}\pa{[0,h_n],\mathsf{H}_{0}}}\leq
C\pa{R}\pa{1+\norm{\alpha}_{L^{\infty}}}$.
Using Lemma \ref{Le: vdif} again, we obtain
\begin{align*}
\norm{I_{3,n}\pa{h_{n}}}_{\mathsf{H}_{0}}
\leq C\pa{R}\pa{1+\norm{\alpha}_{L^{\infty}}}{\norm{\beta}_{L^{\infty}}}h_{n}^{2}.
\end{align*}
We finally set $v^{\pa{4}}\pa{t}=\Phi^{A,n}\pa{h_{n},t}\pa{B\pa{u\pa{t}}-B\pa{u_{n}\pa{t}}}$.
Since
\begin{align*}
v^{\pa{4}}_{t}\pa{t}=&\,i\alpha_{n}\pa{t}\Phi^{A,n}\pa{h_{n},t}A\pa{B\pa{u\pa{t}}-B\pa{u_{n}\pa{t}}}\\
&+\Phi^{A,n}\pa{h_{n},t}\pa{B\pa{u\pa{t}}-{\pa{B\pa{u_{n}\pa{t}}}}}_{t}
\end{align*}
and $u,u_{n}$ are bounded in $X_{T}$, using a similar argument as in previous cases we deduce
the estimate for $I_{4,n}\pa{h_{n}}$. Theorem is thus proven.
\end{proof}

Under hypotheses of Theorem \ref{Th: local convergence} we formulate the result concerning global error estimate.

\begin{theorem}[Global error]
\label{Th: global convergence}
Let $u_{0}\in \mathsf{H}_{1}$ and $T<T^{*}(u_0)$, then
there exists a constant $C>0$ and $n_{0}\in\mathbb{N}$ such that,
for $n\geq n_{0}:$
\[
\max\limits_{0\leq k \leq n}\norm{\Phi\pa{kh_{n}} u_0-\Phi_n
\pa{kh_{n}} u_{0}}_{\mathsf{H}_{0}}\leq C h_{n}.
\]
\end{theorem}
\begin{proof}
Setting
$e_{k}=\norm{\Phi\pa{kh_{n}}u_0-\Phi_{n}\pa{kh_{n}}u_0}_{\mathsf{H}_{0}}$,
it follows that
\begin{align*}
e_{k+1}\leq&\norm{\Phi\pa{h_{n}}\Phi\pa{kh_{n}}u_0-
\Phi\pa{h_{n}}\Phi_{n}\pa{kh_{n}}u_0}_{\mathsf{H}_{0}}+\\
&\norm{\Phi\pa{h_{n}}\Phi_{n}\pa{kh_{n}}{u_0}-\Phi_{n}\pa{h_{n}}\pa{\Phi_n\pa{k
h_n}u_0}}_{\mathsf{H}_{0}}.
\end{align*}
Using estimate \eqref{eq:lipPhiK} and Theorem \ref{Th: local
convergence}, we deduce $e_{k+1}\leq e^{L{h_n}}e_{k}+Ch_{n}^{2}$, from
where, by means of an inductive argument, we conclude the
estimate, valid for $0\leq k \leq n$,
\[
e_{k}\leq Ch_{n}^{2}\sum_{j=0}^{k-1}e^{Lj{h_n}}=
\dfrac{Ch_{n}^{2}}{e^{Lh_{n}}-1}\pa{e^{Lkh_{n}}-1}\leq \dfrac{C\pa{e^{LT}-1}}{L}h_{n}.
\]
This finishes the proof.
\end{proof}
\begin{corollary}
\label{Co: interpolation} Let
$\mathsf{H}_{\theta}=[\mathsf{H}_{0},\mathsf{H}_{1}]_{\theta}$ be
the interpolation Hilbert space, $\theta\in\pa{0,1}$ and
$u_0\in\mathsf{H}_{0}$. If $T<T^{*}$ and $\varepsilon>0$, then
there exists $n_{0}\in\N$ such that
\[
\max\limits_{0\leq k \leq n}\norm{\Phi\pa{kh_{n}} u_0-\Phi_n
\pa{kh_{n}}
u_{0}}_{\mathsf{H}_{\theta}}\leq\varepsilon\,h_{n}^{1-\theta},
\]
holds for $n\geq n_0$.
\end{corollary}

\begin{remark}
Let $u_0, \tilde{u}_0\in\mathsf{H}_{0}$, and let
$T<\min\set{T^*\pa{u_0},T^*\pa{\tilde{u}_0}}$. Using the notation
and the result of Theorem \ref{Th: global convergence}, and the
estimate \eqref{eq:lipPhiK} we deduce
\begin{align*}
\norm{\Phi\pa{kh_{n}}u_0-\Phi_n \pa{kh_{n}}
\tilde{u}_{0}}_{\mathsf{H}_{0}}\leq&
\norm{\Phi\pa{kh_{n}}u_0-\Phi\pa{kh_{n}}
\tilde{u}_{0}}_{\mathsf{H}_{0}}+\\
&\quad \norm{\Phi\pa{kh_{n}}
\tilde{u}_{0}-\Phi_n \pa{kh_{n}} \tilde{u}_{0}}_{\mathsf{H}_{0}}\\
\leq &e^{LT} \norm{u_0-\tilde{u}_{0}}_{\mathsf{H}_{0}}+C h_n.
\end{align*}
\end{remark}
\subsection{Approximation methods}
\label{sub: approx methods}

Assume we can define an approximation $\widetilde{\Phi}^A$ for the flow $\Phi^A$ such that for any $u\in\mathsf{H}_{1}$,  $\norm{\widetilde{\Phi}^A(t)u}_{\mathsf{H}_{1}}\leq C \norm{u}_{\mathsf{H}_{1}}$
and for any $u_0 \in H_1$ and a small time step $h$,
\begin{equation} \label{cotaphiA}
\norm{{\Phi}^A(h) u_{0}-\widetilde{\Phi}^A(h) u_{0}}_{\mathsf{H}_{0}}\leq C h^2\norm{u_{0}}_{\mathsf{H}_{1}}.\end{equation}
Let $\widetilde{\Phi}_h$ be the flow given by \eqref{phitilde}.
From the identity
$\Phi^{A}\pa{t}=\widetilde{\Phi}^{A}\pa{t}+\left(\Phi^{A}\pa{t}-\widetilde{\Phi}^{A}\pa{t}\right)$, we get the following
decomposition for the discrete flow:
$\Phi_{h}=\widetilde{\Phi}_{h}+N_{h}$, where 
\begin{align*}
N_{h}=\sum_{\stackrel{\gamma\in\set{0,1}^{s}}{\gamma\neq 0}}
\prod_{j=1}^{s}\Phi^{B}\pa{b_{j}h}\circ\left(\Phi^{A}\pa{a_{j}h}-\gamma_{j} \widetilde{\Phi}^A\pa{a_j h}\right).
\end{align*}

\begin{proposition}[Approximation method]
\label{prop: approx-method}
Let $\widetilde{\Phi}^A$ be an approximation of the flow $\Phi^A$ satisfying \eqref{cotaphiA}. Let $u_0\in \mathsf{H}_1$, $T<T^*(u_0)$, then
there exists a constant $C>0$ and $n_{0}\in\mathbb{N}$ such that,
for $n\geq n_{0}:$
\[
\max\limits_{0\leq k \leq n}\norm{\Phi\pa{kh_{n}} u_0-\widetilde{\Phi}_{n}^{k} u_0}_{\mathsf{H}_0}\leq C h_n.
\]
\end{proposition}
\begin{proof}

Using that $B:\mathsf{H}_1 \to \mathsf{H}_1$ is Lipchitz with constant $L$, then for all $u\in \mathsf{H}_1$ and for all $s$
\[
\norm{\Phi^B(b_sh)u}_{\mathsf{H}_1}\leq e^{L(b_sh)} \norm{u}_{\mathsf{H}_1},
\]
which combined with inequality \eqref{cotaphiA} yields $\norm{N_{h}u_{0}}_{\mathsf{H}_{0}}\leq Ce^{Lh}h^2\norm{u_{0}}_{{\mathsf{H}_{1}}}$.
Using that
\begin{align*}
\norm{\Phi\pa{h}u_{0}-\widetilde{\Phi}_{h}u_{0}}_{\mathsf{H}_{0}}\leq\norm{\Phi\pa{h}u_{0}-\Phi_{h}u_{0}}_{\mathsf{H}_{0}}
+\norm{\Phi_{h}u_{0}-\widetilde{\Phi}_{h}u_{0}}_{\mathsf{H}_{0}},
\end{align*}
and theorem \eqref{Th: local convergence}, we obtain that there exist $n_0$ such that for $n\geq n_0$
\[
\norm{\Phi\pa{h_n}u_{0}-\widetilde{\Phi}_{h_n}u_{0}}_{\mathsf{H}_{0}}\leq
Ch_n^{2},\]
and therefore we deduce the desired inequality.
\end{proof}

\subsection{Spectral methods}
\label{sub: spectral methods}
We then turn to the discretization in space variables. Let $R>0$
be fixed, let $E$ be the projection valued spectral measure of
$A:\mathsf{H}_{1}\subset D\pa{A}\to\mathsf{H}_{0}$, and let
$P=E([-R,R])$ be the orthogonal projection onto the
$A$-invariant subspace $H=P(\mathsf{H}_0)$.
According to previous subsection, we define $\widetilde{\Phi}^A={\Phi}^A \circ P$ and
$\Phi^{A}\pa{t}=\Phi^{A}\pa{t}\pa{P+I-P}=\widetilde{\Phi}^A\pa{t}+\Phi^{A}\pa{t} \pa{I-P}$. We get the following
decomposition for the discrete flow:
$\Phi_{h}=\widetilde{\Phi}_{h}+N_{h}$, where $h>0$ is a small time
step and
\begin{align*}
N_{h}=\sum_{\stackrel{\gamma\in\set{0,1}^{s}}{\gamma\neq 0}}
\prod_{j=1}^{s}\Phi^{B}\pa{b_{j}h}\circ\Phi^{A}\pa{a_{j}h}P^{1-\gamma_{j}}\pa{I-P}^{\gamma_{j}}.
\end{align*}

\begin{theorem}[Spectral approximation]
\label{th: spec-approx}
Let $u_0\in \mathsf{H}_1$, $T<T^*(u_0)$, and $n\in \N$ be given. Then, for $R>h_n^{-2}=(n/T)^2$ is valid the estimate:
  \begin{align*}
   \max\limits_{0\leq k \leq n}\norm{\Phi\pa{k h_n}u_0-\widetilde{\Phi}_{n}^{k} u_0}_{\mathsf{H}_0}\leq C h_n.
  \end{align*}
\end{theorem}
\begin{proof}
For any $u\in \mathsf{H}_1$ we have
\[
\norm{u-Pu}_{\mathsf{H}_{0}}^{2}=\int_{\abs{\lambda}>R}d\pro{u}{E\pa{\lambda}u}_{\mathsf{H}_{0}}
\leq
R^{-2}\int_{\abs{\lambda}>R}\lambda^{2}\,d\pro{u}{E\pa{\lambda}u}_{\mathsf{H}_{0}}
\]
and then $\norm{u-Pu}_{\mathsf{H}_{0}}\leq R^{-1}\norm{u}_{\mathsf{H}_{1}}$.
Being $\Phi^A$ a unitary operator, we get that $\norm{\Phi^A (I-P)u}_{\mathsf{H}_{0}}\leq R^{-1}\norm{u}_{\mathsf{H}_{1}}$.
Taking $R\geq h_n^{-2}$ we get the desired inequality from proposition \eqref{prop: approx-method}.

\end{proof}

When $\pa{A\pm i}^{-1}$ are compact operators, there exists a
basis $\set{\varphi_{j}}_{j\geq 0}\subset D\pa{A}$ of
$\mathsf{H}_{0}$ and a sequence $\set{\lambda_{j}}_{j\geq
0}\subset\R$ with $\abs{\lambda_{j}}\uparrow\infty$ such that
$A\varphi_{j}=\lambda_{j}\varphi_{j}$. The operator
$\Phi^{A}\pa{t}P$ could be written as
\[
\Phi^{A}\pa{t}Pu=\sum_{\abs{\lambda_{j}}\leq R}e^{-i\lambda_{j}t}\pro{\varphi_{j}}{u}_{\mathsf{H}_{0}}\varphi_{j},
\]
which represents the approximate solution of \eqref{A} in
terms of the eigenfunctions (which in most cases are explicitly
given).

\section{Examples}

\subsection{Nonlinear Schr\"odinger equation}

We consider
\begin{align*}
\begin{cases}
iu_{t}+\Delta u+f(\abs{u}^{2})u+\left(W(x)\ast\abs{u}^2\right) u=0,\\
u\pa{0}=u_{0},
\end{cases}
\end{align*}
where $f:\mathbb{C}\to \mathbb{C}$ is smooth as a real function,
and $W(x)$ is an even function such that $W_1 \in
L^{\infty}\pa{\R^d}, W_2 \in L^{p}\pa{\R^d}$, with $p\geq 2$,
$p>d/4$.  Taking $\mathsf{H}_{1}=H^{s}\pa{\R^{d}}$ and
$\mathsf{H}_{0}=L^{2}\pa{\R^{d}}$, with $s>d/2$, $s\geq 2$,  we
can see that $A=-\Delta$ is a self-adjoint operator, and
$B\pa{u}=-f(\abs{u}^{2})u-\left(W(x)\ast\abs{u}^2\right) u$ is a
locally Lipschitz map (see, Lemmas \ref{ex:f} and \ref{Le:
hartree}). Following these lemmas we can also deduce that, for any
$u_0\in \mathsf{H_1}$, and $T<T^*(u_0)$, the solution verifies
$u\in W^{1,\infty}\pa{[0,T],\mathsf{H_0}};$ in addition, the
nonlinearity $B$ satisfies \eqref{db}. We thus obtain Theorem 4.1
of \cite{BBD} for Lie-Trotter splitting schemes. Using
$H^{2\theta}\pa{\R^{d}}\hookrightarrow L^{\infty}\pa{\R^{d}}$
for $\theta>d/4$ and Corollary \ref{Co: interpolation}, we can see
that
$\norm{u\pa{kh}-u_{n}\pa{kh}}_{L^{\infty}\pa{\R^{d}}}=o\pa{h^{1-\theta}}$.

\begin{remark}
  Since, for $d=3$, the Newtonian potential $W(x)=\abs{x}^{-1}$
  verifies the hypotheses of Lemma \ref{Le: hartree}, the
  convergence results are also valid for the 3-D Schr\"odinger-Poisson
  equation:
  \begin{align*}
    \begin{cases}
      iu_t+\Delta u+V u=0\\
      \Delta V=-|u|^2
    \end{cases}
  \end{align*}
\end{remark}

\begin{remark}
In lower dimensions, $d=1,2$, the kernel $W$ is not bounded and
therefore Lemma \ref{Le: hartree} does not apply. Actually, the
existence of dynamics requires some extra work, see \cite{DLR,M}, mainly connected with a suitable
decomposition of the nonlinearity. However, the conclusions of
Theorem \ref{Th: convergence on H1}- \ref{Th: global convergence}
remain valid but their proofs are more involved.
\end{remark}

\subsection{Gross-Pitaevskii equation with a trapping potential}
We consider the $d$-dimensional initial value problem
\begin{align*}
\begin{cases}
iu_{t}+\Delta u-\Omega u-\abs{u}^{2}u=0,\\
u\pa{0}=u_{0},
\end{cases}
\end{align*}
where $\Omega$ is a positive definite quadratic form. Without loss
of generality we can assume
$\Omega\pa{x}=\omega_{1}^{2}x_{1}^{2}+\cdots+\omega_{d}^{2}x_{d}^{2}$.
This equation is used to describe Bose-Einstein condensates. The
operator $A=-\Delta+\Omega$ has a basis of eigenfunctions
(explicitly) given by
\begin{align*}
\varphi_{\mathbf{k}}\pa{x}=\prod_{j=1}^{d}\varphi_{k_{j}}\pa{\omega_{j}x_{j}}
\end{align*}
for $\mathbf{k}=\pa{k_{1},\ldots,k_{d}}\in\N_{0}^{d}$ with
eigenvalues
$\lambda_{\mathbf{k}}=d+2\sum\limits_{j=1}^{d}k_{j}\omega_{j}^{2}$,
where $\varphi_{k}$ is the $k$-th Hermite function.
In \cite{G} the convergence of a split-step method using Hermite
expansion is studied, the Hilbert spaces
$\tilde{H}^{s}\pa{\R^{d}}=D\pa{A^{s/2}}$ are defined as the
functions $u$ in $L^{2}\pa{\R^{d}}$ such that
$\norm{u}_{\tilde{H}^{s}\pa{\R^{d}}}$ is finite, where
\begin{align*}
\norm{u}_{\tilde{H}^{s}\pa{\R^{d}}}^{2}=\sum_{\mathbf{k}\in\N_{0}^{d}}
\lambda_{\mathbf{k}}^{s}\abs{\pro{\varphi_{\mathbf{k}}}{u}_{L^{2}\pa{\R^{d}}}}^{2}.
\end{align*}
Since $A\geq-\Delta$, we see
$\tilde{H}^{2}\pa{\R^{2}}\hookrightarrow H^{2}\pa{\R^{d}}$, in
particular $\tilde{H}^{2}\pa{\R^{d}}\hookrightarrow
L^{\infty}\pa{\R^{d}}$ if $d\leq 3$. In these cases, Lemma 2 in
\cite{G} implies $D\pa{A}=\tilde{H}^{2}\pa{\R^{3}}$ is an
algebra and then $B\pa{u}=\abs{u}^{2}u$ is a locally Lipschitz
map. Using similar arguments as in the proof of Lemma \ref{ex:f},
we get \eqref{db} for the cubic nonlinearity. Therefore, taking
$\mathsf{H}_{1}=\tilde{H}^{2}\pa{\R^{3}}$ and
$\mathsf{H}_{0}=L^{2}\pa{\R^{3}}$, we obtain the convergence
result given by Theorem \ref{Th: global convergence} and like in
the example above
$\norm{u\pa{kh}-u_{n}\pa{kh}}_{L^{\infty}\pa{\R^{d}}}=o\pa{h^{\theta}}$,
for $\theta<1-d/4$.

\begin{lemma}
\label{Le: h2interseccionL22} For any $u\in D\pa{A}$ the following
estimate do hold:
\[
c^{-1}\pro{Au}{Au}_{L^{2}\pa{\R^{d}}}\leq
\norm{-\Delta u}_{L^{2}\pa{\R^{d}}}^{2}+\norm{\Omega u}_{L^{2}\pa{\R^{d}}}^{2}
\leq c
\pro{Au}{Au}_{L^{2}\pa{\R^{d}}}
\]
with $c=\max\set{2,1+2d^{-2}\sum_{j=1}^{d}\omega_{j}^{2}}$.
\end{lemma}
\begin{proof}
Since $\mathcal{S}\pa{\R^{d}}$ is dense in $D\pa{A}$, we just have to prove
the norm equivalence for any Schwartz function
\begin{align*}
\pro{Au}{Au}_{L^{2}\pa{\R^{d}}}=&
\norm{-\Delta u}_{L^{2}\pa{\R^{d}}}^{2}+\norm{\Omega u}_{L^{2}\pa{\R^{d}}}^{2}
-2\pro{\Delta u}{\Omega u}_{L^{2}\pa{\R^{d}}}
\end{align*}
Using $\pro{\Delta u}{\Omega u}_{L^{2}\pa{\R^{d}}}
=-\pro{\nabla\Omega.\nabla u}{u}-\pro{\Omega\nabla u}{\nabla u}$, we get
\begin{align*}
2\pro{\Delta u}{\Omega u}_{L^{2}\pa{\R^{d}}}\leq&
-2\pro{\nabla\Omega.\nabla u}{u}_{L^{2}\pa{\R^{d}}}
=\pro{\Delta\Omega\,u}{u}_{L^{2}\pa{\R^{d}}}\\
=&2\sum_{j=1}^{d}\omega_{j}^{2}\norm{u}_{L^{2}\pa{\R^{d}}}^{2}.
\end{align*}
Since $\pro{Au}{Au}_{L^{2}\pa{\R^{d}}}\geq d^{2}\norm{u}_{L^{2}\pa{\R^{d}}}^{2}$,
we have
\begin{align*}
\norm{-\Delta u}_{L^{2}\pa{\R^{d}}}^{2}+\norm{\Omega u}_{L^{2}\pa{\R^{d}}}^{2}
\leq\pa{1+2d^{-2}\textstyle{\sum_{j=1}^{d}}\omega_{j}^{2}}
\pro{Au}{Au}_{L^{2}\pa{\R^{d}}}.
\end{align*}
From $2\pro{\Delta u}{\Omega u}_{L^{2}\pa{\R^{d}}}\leq
\norm{-\Delta u}_{L^{2}\pa{\R^{d}}}^{2}
+\norm{\Omega u}_{L^{2}\pa{\R^{d}}}^{2}$, we obtain
\[
\pro{Au}{Au}_{L^{2}\pa{\R^{d}}}\leq 2\norm{-\Delta u}_{L^{2}\pa{\R^{d}}}^{2}
+2\norm{\Omega u}_{L^{2}\pa{\R^{d}}}^{2}
\]
and then, the lemma follows.
\end{proof}

\begin{corollary}
For $d\leq 3$, $\tilde{H}^{2}\pa{\R^{d}}$ is an algebra with the
pointwise product.
\end{corollary}
\begin{proof}
From the estimate $\norm{\Omega uv}_{L^{2}\pa{\R^{d}}}\leq
\norm{\Omega u}_{L^{2}\pa{\R^{d}}}
\norm{v}_{L^{\infty}\pa{\R^{d}}}$ and the embedding
$\tilde{H}^{2}\pa{\R^{d}}\hookrightarrow L^{\infty}\pa{\R^{d}}$, we obtain
$\norm{\Omega uv}_{L^{2}\pa{\R^{d}}}\leq C
\norm{u}_{\tilde{H}^{2}\pa{\R^{d}}}
\norm{v}_{\tilde{H}^{2}\pa{\R^{d}}}$.
Using $-\Delta\pa{uv}=-\Delta u\,v-u\Delta v-2\nabla u.\nabla v$, we have
\begin{align*}
\norm{-\Delta\pa{uv}}_{L^{2}\pa{\R^{d}}}\leq &
\norm{-\Delta u}_{L^{2}\pa{\R^{d}}}
\norm{v}_{L^{\infty}\pa{\R^{d}}}+
\norm{-\Delta v}_{L^{2}\pa{\R^{d}}}
\norm{u}_{L^{\infty}\pa{\R^{d}}}\\
&+2\norm{\nabla u}_{L^{4}\pa{\R^{d}}}\norm{\nabla v}_{L^{4}\pa{\R^{d}}}.
\end{align*}
Since
\begin{equation}
\label{embL4}
\begin{split}
\norm{\nabla u}_{L^{4}\pa{\R^{d}}}^{2}\leq &
C\norm{u}_{L^{2}\pa{\R^{d}}}^{\pa{4-d}/4}
\norm{-\Delta u}_{L^{2}\pa{\R^{d}}}^{\pa{4+d}/4}\\
\leq & C\pa{\norm{u}_{L^{2}\pa{\R^{d}}}^{2}+
\norm{-\Delta u}_{L^{2}\pa{\R^{d}}}^{2}}
\leq C\norm{u}_{\tilde{H}^{2}\pa{\R^{d}}}^{2},
\end{split}
\end{equation}
we get $\norm{-\Delta\pa{uv}}_{L^{2}\pa{\R^{d}}}\leq
C\norm{u}_{\tilde{H}^{2}\pa{\R^{d}}}
\norm{v}_{\tilde{H}^{2}\pa{\R^{d}}}$ and
\[
\norm{uv}_{\tilde{H}^{2}\pa{\R^{d}}}\leq C
\norm{u}_{\tilde{H}^{2}\pa{\R^{d}}}
\norm{v}_{\tilde{H}^{2}\pa{\R^{d}}},
\]
this finishes the proof.
\end{proof}

\begin{proposition}
Let $f$ be as in example \ref{ex:f} and $d\leq 3$, then
the map $u\mapsto f\pa{u}$ is bounded and locally Lipschitz on
$\tilde{H}^{2}\pa{\R^{d}}$.
\end{proposition}
\begin{proof}
Let $R>0$ such that $\norm{u}_{L^{\infty}\pa{\R^{d}}}\leq R$,
since $\abs{f\pa{u}}\leq C\abs{u}$ if $\abs{u}\leq R$, we have
$\norm{\Omega f\pa{u}}_{L^{2}\pa{\R^{d}}}\leq
C\norm{\Omega u}_{L^{2}\pa{\R^{d}}}$.
Using that
$\Delta f\pa{u}=f''\pa{u}\abs{\nabla u}^{2}+f'\pa{u}\Delta u$,
we obtain
\begin{align*}
\norm{-\Delta f\pa{u}}_{L^{2}\pa{\R^{d}}}^{2}
+\norm{\Omega f\pa{u}}_{L^{2}\pa{\R^{d}}}^{2}\leq&C\left(
\norm{-\Delta u}_{L^{2}\pa{\R^{d}}}^{2}\right.\\
&\left.+\norm{\Omega u}_{L^{2}\pa{\R^{d}}}^{2}
+\norm{\nabla u}_{L^{4}\pa{\R^{d}}}^{2}\right),
\end{align*}
from \eqref{embL4} and Lemma \ref{Le: h2interseccionL22} we have
\begin{align*}
\pro{Af\pa{u}}{Af\pa{u}}_{L^{2}\pa{\R^{d}}}\leq & C
\norm{-\Delta f\pa{u}}_{L^{2}\pa{\R^{d}}}^{2}
+\norm{\Omega f\pa{u}}_{L^{2}\pa{\R^{d}}}^{2}\\
\leq &
C\pa{\norm{u}_{L^{2}\pa{\R^{d}}}^{2}+
\norm{-\Delta u}_{L^{2}\pa{\R^{d}}}^{2}}
\leq C\pro{Au}{Au}_{L^{2}\pa{\R^{d}}}.
\end{align*}
Let $u,v\in\tilde{H}^{2}\pa{\R^{d}}$ such that
$\norm{u}_{\tilde{H}^{2}\pa{\R^{d}}},\norm{u}_{\tilde{H}^{2}\pa{\R^{d}}}\leq R$, then
\begin{align*}
\norm{f\pa{u}-f\pa{v}}_{\tilde{H}^{2}\pa{\R^{d}}}\leq&
\int_{0}^{1}\norm{f'\pa{\pa{1-t}u+tv}}_{\tilde{H}^{2}\pa{\R^{d}}}
\norm{u-v}_{\tilde{H}^{2}\pa{\R^{d}}}dt\\
\leq & C\norm{u-v}_{\tilde{H}^{2}\pa{\R^{d}}},
\end{align*}
which expresses that $f$ is a locally Lipschitz map.
\end{proof}

Using similar arguments as those used in the proof of Lemma
\ref{ex:f}, we can see that the nonlinear local term given by
$B\pa{u}=f(\abs{u}^{2})u$ verifies \eqref{db} and then the
conclusion of Theorem \ref{Th: global convergence} holds.

\subsection{Nonlinear wave interaction model}
Consider the system of evolution equations modelling wave-wave
interaction in quadratic nonlinear media (see \cite{BDDR} and references therein). This model describes the
nonlinear and nonlocal cross-interaction of two waves in $1+1$
dimensions. The interaction is described by nonlocal (integral)
expressions:
\begin{align*}
\begin{cases}
u^{\pa{1}}_{t}-u^{\pa{1}}_{x}+\nu\,g\,u^{\pa{2}}=0,\\
u^{\pa{2}}_{t}+u^{\pa{2}}_{x}-\nu\,g^{*}\,u^{\pa{1}}=0,\\
u^{\pa{1}}\pa{0}=u^{\pa{1}}_{0},\, u^{\pa{2}}\pa{0}=u^{\pa{2}}_{0},
\end{cases}
\end{align*}
where $\nu=\pm1$ and $g_{x}=u^{\pa{2}*}\,u^{\pa{1}}$, $g\pa{x}\to 0$ when $x\to-\infty$.
Consider the spaces $\mathsf{H}_{1}=H^{1}\pa{\R}\times H^{1}\pa{\R}$, $\mathsf{H}_{0}=L^{2}\pa{\R}\times L^{2}\pa{\R}$ and
the operator $A=i\partial_{x}\sigma_{z}$.
Define $B\pa{u}=\nu g\pa{u}\sigma_{y}.u$, with
\[
g\pa{u}\pa{x,t}=\int_{-\infty}^{x}u^{\pa{2}*}\pa{y,t}u^{\pa{1}}\pa{y,t}dy
\]
and $\sigma_{y},\sigma_{z}$ the Pauli matrices. Taking
\[
\pa{g'\pa{u}w}\pa{x,t}=\int_{-\infty}^{x}\pa{w^{\pa{2}*}\pa{y,t}u^{\pa{1}}\pa{y,t}+
u^{\pa{2}*}\pa{y,t}w^{\pa{1}}\pa{y,t}}dy,
\]
we can see that $B'\pa{u}w=\nu g'\pa{u}w\,\sigma_{y}.u+\nu
g\pa{u}\sigma_{y}.w$. From Cauchy inequality, we get
$\norm{g'\pa{u}w}_{L^{\infty}\pa{\R}}\leq\norm{u}_{L^{2}\pa{\R}}\norm{w}_{L^{2}\pa{\R}}$.
From the expression of $B'\pa{u}w$, we conclude
$\norm{B'\pa{u}w}_{L^{2}\pa{\R}}\leq
C\norm{u}_{L^{2}\pa{\R}}^{2}\norm{w}_{L^{2}\pa{\R}}$. Then,
\eqref{db} is verified and therefore the conclusions of Theorem
\ref{Th: local convergence} and Theorem \ref{Th: global
convergence} are valid.

As an application of these results, we study the behavior of
solutions with compact support. If
$\mathrm{supp}\pa{u_{0}}\subset\pa{a,b}$, since $A$ is a first
order linear wave equations and it holds
$\mathrm{supp}\pa{B\pa{u}}\subset\mathrm{supp}\pa{u}$, it follows
that $\mathrm{supp}\pa{\Phi^{A}\pa{t}u_{0}}\subset\pa{a-t,b+t}$
and $\mathrm{supp}\pa{\Phi^{B}\pa{t}u}\subset\mathrm{supp}\pa{u}$.
Therefore, $\mathrm{supp}\pa{u_{n}\pa{t}}\subset\pa{a-t,b+t}$
which implies $\mathrm{supp}\pa{u\pa{t}}\subset\pa{a-t,b+t}$.

\section{Numerical example}

Consider de Schr\"odinger--Poisson equation in $\mathbb{T}$, i.e. $u$ is a $1$--periodic solution of
\begin{align}
\label{eq: S-P periodic}
\begin{cases}
iu_{t}+u_{xx}+\abs{u}^{2}u+Vu=0,\\
V_{xx}=\mathcal{D}-\abs{u}^{2},\\
u\pa{0}=u_{0},
\end{cases}
\end{align}
where $\mathcal{D}\in C^{\infty}\pa{\mathbb{T}}$ is a given real--valued function. We assume that neutrality condition is verified:
\[
 \int_{\mathbb{T}}\mathcal{D}\pa{x}dx=\norm{u_{0}}_{L^{2}\pa{\mathbb{T}}}^{2},
\]
since $\norm{u\pa{t}}_{L^{2}\pa{\mathbb{T}}}^{2}$ is a conserved quantity, this condition holds for any $t$.
The potential $V$ can be calculated by $V=-G*\varrho$, where $\varrho=\mathcal{D}-\abs{u}^{2}$ and $G$ is the Green potential defined as the
$1$--periodic function such that $G\pa{x}=x\pa{1-x}/2$ on $[0,1]$.
We consider $\mathsf{H}_{0}=L^{2}\pa{\mathbb{T}}$, $\mathsf{H}_{1}=H^{2}\pa{\mathbb{T}}$,
defining the self--adjoint operator $A=-\partial_{xx}$ and
\[
B\pa{u}=-\abs{u}^{2}u+\pa{G*\varrho}u,
\]
we can write \eqref{eq: S-P periodic} in the form \eqref{ec-evol} and from Lemma \ref{Le: hartree}, $B$ verifies \eqref{db}.

The linear flow $\Phi^{A}$ can be written as ${\pa{\Phi^{A}\pa{t}u}}{(x)}=\sum_{p\in\Z}\hat{u}_{p}e^{-i4\pi^{2}p^{2}t}e^{i2\pi px}$, where
\[
 \hat{u}_{p}=\int_{\mathbb{T}}u\pa{x}e^{{-}i2\pi px}dx.
\]
Let $w$ be the solution of \eqref{eq: flow B} with $w\pa{0}=u$, using $V$ is a real--valued potential,
we can see that $\Re{w^{*}w_{t}}=0$,
which implies $\abs{w}=\abs{u}$ and then $V$ is constant in $t$.
Therefore $\Phi^{B}\pa{t}u=e^{it\pa{V+\abs{u}^{2}}}u$, where $V$ is calculated using $u$.
Observe that if $\varrho=\mathcal{D}-\abs{u}^{2}$, then it holds $\hat{\varrho}_{0}=0$ and the potential can be expanded by
$V\pa{x}=-\sum_{p\in\Z}\hat{\varrho}_{p}\pa{2\pi p}^{-2}e^{i2\pi px}$.


\subsection{Solving by Discrete Fourier Transform}

We show a numerical method using discrete Fourier coefficients. Let $m$ be the odd integer $m=2l+1$
and consider $\pa{I_{m}u}\pa{x}=\sum\limits_{p=-{l}}^{{l}}\hat{U}_{p}e^{i2\pi px},$ where
$\hat{U}_{p}$ is the discrete Fourier coefficient given by
\[
\hat{U}_{p}=\frac{1}{m}\sum_{q=0}^{m-1}U_{q}e^{-i2\pi pq/m}
\]
and $U_{q}=u\pa{q/m}$.
Since $e^{-i2\pi pq/m}=e^{-i2\pi q\pa{p\pm m}/m}$, we have $\hat{U}_{p}=\hat{U}_{p\pm m}$. We also know that
\[
U_{q}=\sum_{p=0}^{m-1}\hat{U}_{p}e^{i2\pi pq/m}.
\]
It is known that $\norm{u-I_{m}u}_{L^{2}\pa{\mathbb{T}}}\leq C m^{-2}\norm{u}_{H^{2}\pa{\mathbb{T}}}$ (see Lemma 2.2 in \cite{T}) and then we have
\begin{proposition}
Let $\Phi_{m}^{A}\pa{t}=\Phi^{{A}}\pa{t}I_{m}$, for any $u\in H^{2}\pa{\mathbb{T}}$ it is verified
\begin{align*}
\norm{\Phi^{A}\pa{t}u-\Phi_{m}^{A}\pa{t}u}_{L^{2}\pa{\mathbb{T}}}\leq C m^{-2}\norm{u}_{H^{2}\pa{\mathbb{T}}}.
\end{align*}
\end{proposition}
\noindent
We can see $\Phi_{m}^{A}\pa{t}$ is an approximation of the flow $\Phi^A$ that verifies inequality \eqref{cotaphiA} in subsection \ref{sub: approx methods} for $m\geq n$.
From definition of $\Phi_{m}^{A}\pa{t}$ and $\hat{U}_{p}=\hat{U}_{p\pm m}$, it holds
\begin{align*}
\pa{\Phi_{m}^{A}\pa{t}u}\pa{q/m}=&\,\sum_{p=-l}^{l}\hat{U}_{p}e^{-i4\pi^{2}p^{2}t}e^{i2\pi pq/m}\\
=&\,\sum_{p=l+1}^{m-1}\hat{U}_{p}e^{-i4\pi^{2}\pa{m-p}^{2}t}e^{i2\pi pq/m}
+\sum_{p=0}^{l}\hat{U}_{p}e^{-i4\pi^{2}p^{2}t}e^{i2\pi pq/m}\\
=&\,\sum_{p=0}^{m-1}\hat{U}_{p}e^{-i\lambda_{p}t}e^{i2\pi pq/m},
\end{align*}
where $\lambda_{p}=4m^{2}\pi^{2}h\pa{p/m}$ for $0\leq p\leq m-1$ and $h\pa{\nu}=\nu^{2}-2\pa{\nu-1/2}_{+}$.

The solution of \eqref{eq: flow B} can be exactly calculated as
\[
\pa{\Phi^{B}\pa{t}u}\pa{q/m}=e^{it\pa{V_{q}+N_{q}}}U_{q},
\]
where $N_{q}=\abs{U_{q}}^{2}$and the potential $V$ is given by
\[
V_{q}=-\sum_{p=1}^{m-1}\hat{\varrho}_{p}\lambda_{p}^{-2}e^{i2\pi pq/m},
\]
with $\hat{\varrho}_{p}=\hat{\mathcal{D}}_{p}-\hat{N}_{p}$. Observe that
the neutrality condition reads as
$\hat{\varrho}_{0}=\hat{\mathcal{D}}_{0}-\hat{N}_{0}=0$.
Therefore, the Lie--Trotter algorithm can be written as:
\texttt{\begin{itemize}
\item[-] Fix $n$.
\item[-] Asign $h=T/n$.
\item[-] Fix $m\sim h^{-1}$.
\item[-] Transform $D$ to $\hat{D}$ using FFT.
\item[-]  Compute $\lambda^{-2}$.
\item[-] Compute $\exp\pa{-i\lambda h}$.
\item[-] Evaluate $U=u_{0}\pa{q/m}$ for $q=0,\ldots,m-1$.
\item[-] For $k=1,\ldots,n$ do
\begin{enumerate}[1.]
\item Transform $U$ to $\hat{U}$ using FFT $\pa{m\times\log\pa{m}\texttt{ ops}}$.
\item Multiple $\hat{U}$ by $\exp\pa{-i\lambda h}$ $\pa{m\texttt{ ops}}$.
\item Obtain $U^{\pa{A}}$ anti-transforming FFT $e^{-i\lambda h}.\hat{U}$ $\pa{m\times\log\pa{m}\texttt{ ops}}$.
\item Compute $N=\abs{U^{\pa{A}}}^{2}$ $\pa{m\texttt{ ops}}$.
\item Transform $N$ to $\hat{N}$ using FFT $\pa{m\times\log\pa{m}\texttt{ ops}}$.
\item Compute $\hat{\varrho}$ substracting $\hat{N}$ from $\hat{D}$.
\item Multiple $\hat{\varrho}$ by $\lambda^{-2}$ $\pa{m\texttt{ ops}}$.
\item Obtain $V$ anti-transforming FFT $-\lambda^{-2}.\hat{\varrho}$ $\pa{m\times\log\pa{m}\texttt{ ops}}$.
\item Sum $N$ and $V$.
\item Evaluate $\exp\pa{ih\pa{V+N}}$ $\pa{b\times m\texttt{ ops}}$.
\item Obtain $U$ multiplying $\exp\pa{ih\pa{V+N}}.U^{\pa{A}}$ $\pa{m\texttt{ ops}}$.
\item Asign $U[k]=U$.
\end{enumerate}
\end{itemize}
}
\noindent
The computational cost is proportional to $n\times m\times\log\pa{m}$.


To illustrate Theorem \eqref{Th: global convergence} we present a numerical experiment in one space dimension.
We use the algorithm described above to discretize the Schr\"odinger--Poisson  equation \eqref{eq: S-P periodic} with initial data
$u_0(x)=\sin^{\frac{3}{2}+\alpha}(\pi x)$ with $\alpha >0$ small so that $u_0\in H^2$ but $u_0 \notin H^{2+s}$ for $s>\alpha$,
and $\mathcal{D}\pa{x}=\gamma\pa{\alpha}\pa{1+\pa{1+16 \pi^{2}}} \cos\pa{4 \pi  x}$, with
\[
\gamma\pa{\alpha}=\frac{\Gamma (\alpha +2)}{\sqrt{\pi }\,\Gamma
   \left(\alpha +\frac{5}{2}\right)}.
\]
Figure $1$ shows the order dependence of the $L^\infty$ error at time $T = 1$ on the time step-size h.
The calculations are performed with a space discretization of $2 \times 10^5+1$
and compared to the result with a time step-size $h = \frac{10^{-5}}{2}$.
\begin{figure}[!h]\includegraphics[width=9cm, height=6cm]{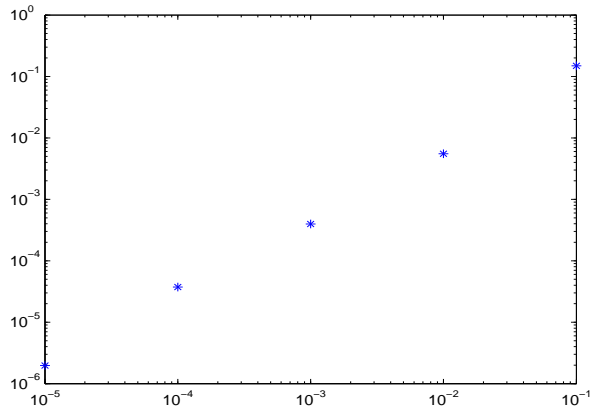}
\caption{Discretization error for different time step}
\end{figure}

Figure 2 illustrates the dependence of the $L^\infty$ error on the space discretization
parameter $n$. Here, we use a fixed time step-size $h=10^{-3}$ and compare
the results with the result for $n = 2^{14}+1$.

\begin{figure}[!h] \includegraphics[width=9cm, height=6cm]{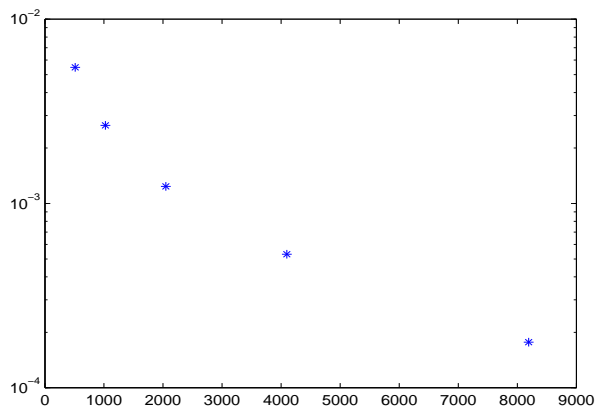}
\caption{Discretization error for different space discretization}
\end{figure}


\section*{Acknowledgment}This work has been supported in part by PIP11420090100165, CONICET and
MATH-Amsud 11MATH-02, IMPA/CAPES (Brazil)--MINCYT (Argentina)--CNRS/INRIA (France)


\end{document}